\documentclass[12pt]{amsart}
\usepackage{tikz-cd}
\usepackage[displaymath, mathlines, pagewise]{lineno}

\usepackage{fullpage}

\usepackage{amssymb,latexsym}
\usepackage{amsmath,amsfonts,amsthm,amscd,amsxtra,enumerate}
\usepackage{mathtools}
\usepackage[all]{xy}
\usepackage{amsmath,latexsym,amssymb, enumerate, amsthm, epsfig, hyperref} 

\usepackage[margin=1in]{geometry}
\usepackage{epsfig}
\usepackage{[;''''''''''''youngtab}
\usepackage{young}
\usepackage[vcentermath]{youngtab}
\usepackage{young}
\usepackage{youngtab}
\usepackage{ytableau}
\usepackage{multicol}

\renewcommand{\small}{\fontsize{10}{10}\selectfont}

\newtheorem{theorem}{Theorem}[section]

\DeclareMathOperator{\Grass}{Grass}
\DeclareMathOperator{\Flag}{Flag}
\DeclareMathOperator{\rank}{rank}

\DeclareMathOperator{\Tor}{Tor}

\DeclareMathOperator{\im}{im}

\DeclareMathOperator{\depth}{depth}

\theoremstyle{plain}

\usepackage{geometry,mathtools}
\date{\today}

\subjclass[2010]{13C40, 13D02, 13H10} 
\keywords{perfect ideal, codimension three, linkage class, generic ring, Cohen-Macaulay, complete intersection}

\begin{document}

\title{\textbf{The family of perfect ideals of codimension 3, of type 2 with 5 generators}}

\author[Ela Celikbas]{Ela Celikbas}
\address{Department of Mathematics, West Virginia University, Morgantown, WV 26506.}
\email{ela.celikbas@math.wvu.edu}

\author[Jai Laxmi]{Jai Laxmi}
\address{Department of Mathematics, I.I.T. Bombay, Powai, Mumbai 400076.}
\email{jailaxmi@math.iitb.ac.in}

\author[Witold Kra\'skiewicz]{Witold Kra\'skiewicz}
\address{Department of Mathematics, Nicholas Copernicus University, Toru\'n, Poland.}
\email{wkras@mat.umk.pl}

\author{Jerzy Weyman}
\address{Department of Mathematics, University of Connecticut, Storrs, CT 06269.}
\email{jerzy.weyman@uconn.edu}

%
%
%


\maketitle
\begin{abstract}
In this paper we define an interesting family of  perfect ideals  of codimension three, with five generators, of Cohen-Macaulay type two with trivial multiplication on the $\Tor$ algebra.
This family is likely to play a key role in classifying  perfect ideals  with five generators of type two.

\end{abstract}
\section{Introduction}

Perfect ideals of codimension three have been investigated for a long time. Linkage theory (\cite{PS}) suggests such ideals might be possible to classify. Indeed, if one applies minimal linkage to a perfect ideal of codimension three, one gets an ideal with a minimal free resolution with the same sum of Betti numbers as the original one, and after double link one obtains the ideal with the free resolution with modules of the same ranks as for the original ideal.

Since the Buchsbaum-Eisenbud discovery of the structure of Gorenstein ideals of codimension three (\cite{DABDEs77}), one can obtain via linkage also the structure of perfect ideals of codimension three with four generators (almost complete intersections) (see \cite{AEB87}, \cite{CVW18}). Hence the smallest class of perfect ideals of codimension three, whose structure is not known are those with five generators, of Cohen-Macaulay type two. Such ideals were classified via linkage with additional assumption that  one of the Koszul relations on the generators is also one of minimal generators of the first syzygy module of the ideal (\cite{AEB87}). Thus the next step is to deal with the ideals for which it does not happen.

It is conceivable that such ideals are in the linkage class of complete intersections, as one can prove (see \cite{CVW-2}) that the grading obstruction of Huneke-Ulrich (\cite{HU}) does not occur for resolutions of this type (at least for homogeneous ideals in a polynomial ring of three variables.

One of the difficulties in classifying such ideals is that most of the known examples are constructed via Macaulay inverse systems so there are many of them corresponding to different Hilbert functions and it is hard to see whether they come from few ``generic" families. The smallest class of examples known to authors was the class of ideals in the polynomial ring with three variables, obtained by Macaulay inverse system from two forms of degree four.

In this paper we define an interesting family of the ideals in question, over polynomial ring in 17 variables. We prove its basic properties: the multiplicative structure on their resolutions has constants of positive degree, but they are in the linkage class of complete intersections. 

The family of ideals we describe was discovered using the results of \cite{JWm16}. In that paper the author constructed the so-called generic rings ${\hat R}_{gen}$ for resolutions of length three for all formats. However, the ring ${\hat R}_{gen}$ is Noetherian only in very few cases. Our format of five generators and type two is one of those cases. The family we got comes from picking certain slice in the spectrum of the ring ${\hat R}_{gen}$. Analogous construction for the formats $(1,n,n,1)$ with $n$ even and for the formats $(1,4,n+3,n)$ gives the generic perfect ideals for these formats. Based on this, we expect our family to play a key role in classifying the perfect ideals with five generators, of Cohen-Macaulay type two.

\section{A family of perfect Ideals of codimension three, of type two, with five generators. }
Let $K$ be a field of characteristics different from two. Let $F$ and $G$ be two vector spaces over $K$ of dimensions four and two, respectively. Consider the affine space
$$X=\bigwedge\limits^{2} F^*\otimes G^*\oplus \bigwedge\limits^3 F^*$$
whose elements are pairs consisting of a pencil of $4\times 4$ skew-symmetric matrices
and a $4$-vector. The coordinate ring $A$ of $X$ can be identified with the symmetric algebra 
$$A=\text{Sym}(\bigwedge\limits^{2} F\otimes G\oplus \bigwedge\limits^3 F).$$

We denote the coordinate functions in the coordinate ring $A$ by $x_{i,j}$, $y_{i,j}$ $(1\leq i<j\leq 4)$, and $z_{i,j,k}$ $(1\leq i<j<k\leq 4)$. $A$ is a bigraded ring with $\deg(x_{i,j})=\deg(y_{i,j})=(1,0)$, and $\deg(z_{i,j,k})=(0,1)$.

Consider the equivariant ideal $J$ in $A$ generated by the representation $S_{2,2,2,1}F\otimes \bigwedge\limits^{2} G$ in bidegree $(2,1)$ and the representation $S_{2,2,2,2}F\otimes S_{2,2}G$ in bidegree $(4,0)$. We describe generators of $J$ explicitly. We use $\Delta(ij,kl)$ to denote the $2\times 2$ minor of the matrix 
$$\begin{bmatrix*}
x_{1,2} & x_{1,3} & x_{1,4} & x_{2,3} & x_{2,4} & x_{3,4}\\
y_{1,2} & y_{1,3} & y_{1,4} & y_{2,3} & y_{2,4} & y_{3,4}
\end{bmatrix*} $$
corresponding to the columns labeled by $(i,j)$ and $(k,l)$. \\

The cubic generators of bidegree $(2,1)$ are $u_{1,2,3}$, $u_{1,2,4}$, $u_{1,3,4}$ and $u_{2,3,4}$ where
{\small
\begin{align*}
u_{1,2,3}&=-2z_{2,3,4}\Delta(12,13)+2z_{1,3,4}\Delta(12,23)-2z_{1,2,4} \Delta(13,23)+z_{1,2,3}(\Delta(13,24)-\Delta(12,34)+\Delta(14,23))\\
u_{1,2,4}&=2z_{2,3,4}\Delta(12,14)-2z_{1,3,4}\Delta(12,24)+z_{1,2,4}(\Delta(12,34)+\Delta(13,24)+\Delta(14,23))- 2z_{1,2,3}\Delta(14,24)\\
u_{1,3,4}&=2z_{2,3,4}\Delta(13,14)+z_{1,3,4}(-\Delta(12,34)-\Delta(13,24)+\Delta(14,23))+2z_{1,2,4}\Delta(13,34)- 2z_{1,2,3}\Delta(14,34)\\
u_{2,3,4}&=z_{2,3,4}(-\Delta(12,34)-\Delta(13,24)-\Delta(14,23))-2z_{1,3,4}\Delta(23,24)+2z_{1,2,4}\Delta(23,34)- 2z_{1,2,3}\Delta(24,34).
\end{align*}}

The generator of degree $(4,0)$  is a discriminant of the Pfaffian of the skew-symmetric
$4\times4$ matrix of linear forms treated as a binary form, so it equals $u=b^2-4ac$ where 
\begin{align*}
a &=x_{1,2}x_{3,4}-x_{1,3}x_{2,4}+x_{1,4}x_{2,3},\\
b &=x_{1,2}y_{3,4}-x_{1,3}y_{2,4}+x_{1,4}y_{2,3}+x_{3,4}y_{1,2}-x_{2,4}y_{1,3}+x_{2,3}y_{1,4},\\
c &=y_{1,2}y_{3,4}-y_{1,3}y_{2,4}+y_{1,4}y_{2,3}.
\end{align*}

Let $u_j=\sum_{i\neq j}(-1)^i x_{i,j} z_{\hat{i}}$, $v_j=\sum_{i\neq j}(-1)^iy_{i,j}z_{\hat{i}}$, $\delta_1=\Delta(12,34)$, $\delta_2=\Delta(13,24)$, and $\delta_3=\Delta(14,23)$.
Then the minimal resolution of $A/J$ over $A$ is 
\[{\Bbb F}_\bullet:0\xrightarrow{}A^2\xrightarrow{d_3} A^6\xrightarrow{d_2}A^5\xrightarrow{d_1}A \;\;\text{where}\]
\begin{align*}d_1&=\begin{bmatrix*}
u_{2,3,4} & u_{1,3,4} & u_{1,2,4} & u_{1,2,3} & u
\end{bmatrix*}, \\
d_2&=\begin{bmatrix*}
v_1 & u_1 & -\delta_1+\delta_2-\delta_3 & 2\Delta(13,14) & -2\Delta(12,14) & -2\Delta(12,13) \\
-v_2 & -u_2 &-2\Delta(23,24) & -\delta_1-\delta_2+\delta_3 & 2\Delta(12,24) & 2 \Delta(12,23) \\
v_3 & u_3 & 2\Delta(23,34)  & 2\Delta(13,34) & -\delta_1-\delta_2-\delta_3 & -2\Delta(13,23)\\
v_4 & u_4 &2\Delta(24,34) & 2\Delta(14,34) & -2\Delta(14,24) &\delta_1-\delta_2-\delta_3\\
0 & 0&-z_{2,3,4} & -z_{1,3,4} & z_{1,2,4} & z_{1,2,3}
\end{bmatrix*},\\
d_3&=\begin{bmatrix*}
b & 2a\\
-2c & -b\\
-v_1& -u_1\\
v_2 & u_2\\
v_3 & u_3\\
-v_4 & -u_4
\end{bmatrix*}.
\end{align*}

\begin{proof}
We use Buchsbaum-Eisenbud exactness criterion. The rank conditions are obviously satisfied. So we need to show that for $1\leq i\leq 3$, $\depth(I(d_i))\geq 3$. 
In fact, it is enough to show that the codimension of the ideal $J$ is $3$ and that all the ideals $I(d_i)$ for $1\leq i\leq 3$ are the same up to radical.
In order to see it, we give a geometric interpretation of the ideal $J$. 
 The element $u$ is the hyperdiscriminant for the representation $\bigwedge^2 F^*\otimes G^*$. Vanishing of this polynomial on a pencil $s(x_{i,j})+t(y_{i,j})$ means that in this pencil there is only one (``double") member of rank $\le 2$. The zero set $V(J)$ consists of pairs $((x_{i,j}, y_{i,j}), (z_{i,j,k}))$ such that the trivector $(z_{i,j,k})$ viewed as a functional on $F^*$ vanishes on the kernel of rank two member of the pencil.

The resolution with differentials $d_3, d_2, d_1$ given above can be constructed by a geometric method (see \cite{JWm03}) as follows.
Let us first write it in terms of representation theory 
$${\Bbb F}_\bullet: 0\rightarrow (\bigwedge^4F)^{\otimes 4}\otimes (\bigwedge^2G)^{\otimes 2}\otimes G\otimes A(-5,-2)\rightarrow $$
$$\rightarrow (\bigwedge^4F)^{\otimes 3}\otimes \bigwedge^2G\otimes G\otimes A(-3,-2)
\oplus (\bigwedge^4F)^{\otimes 2}\otimes\bigwedge^3 F\otimes (\bigwedge^2G)^{\otimes 2}\otimes A(-4,-1)\rightarrow$$
$$\rightarrow  (\bigwedge^4F)^{\otimes 2}\otimes (\bigwedge^2G)^{\otimes 2}\otimes A(-4,0)\oplus
\bigwedge^4F\otimes\bigwedge^3 F\otimes \bigwedge^2 G\otimes A(-2,-1)\rightarrow A$$

\medskip
The desingularization $Z$ of our set $V(J)$ is given by the homogeneous bundle $\eta^*$ where 
$\eta$ has weights $(1,1,0,0;1,0)$, $(1,0,1,0;1,0)$, $(0,1,1,0;1,0)$, $(1,0,0,1;1,0)$, $(1,1,0,0;0,1)$, $(1,0,1,0;0,1)$, $(1,1,1,0;0,0)$.
This is a homogeneous bundle of rank $7$ over $\Flag(1,3;F)\times \Grass(1,G)$, so its rank is $7+5+1=13$ as needed. 
One checks that the pushdown of the Koszul complex gives our complex ${\Bbb F}_\bullet$.
Even better is to consider bundle $\xi$ with the weights $(0,0,1,1;0,1)$, $(0,1,0,1;0,1)$, $(1,0,0,1;0,1)$, $(0,1,1,0;0,1)$, $(1,0,1,0;0,1)$,
$(0,0,1,1;1,0)$, $(0,1,1,1;0,0)$, $(1,0,1,1,0,0)$. This is a bundle desingularizing discriminant on the first coordinate and a two-dimensional sub bundle on the second. The homogeneous space is $\Grass(2,F)\times \Grass(1,G)$. So dimension of $Z$ is $8+4+1=13$.
Pushing down the corresponding Koszul complex gives our complex ${\Bbb F}_\bullet$. It follows from the results of \cite{JWm03}, chapter 5, that the variety $V(J)$ has rational singularities so $A/J$ is Cohen-Macaulay of codimension three.
The rest of the claims follow.
\end{proof}


\section{The deformed ideal J({\rm t})}
In this section, we give a deformed ideal $J(t)$ which is an ideal in the bigger polynomial ring $B=A[t]$. Equivariantly, the variable $t$ is just $\bigwedge\limits^4 F\otimes \bigwedge\limits^2 G$ and we set $\deg(t) = 2$. If we let $u_j=\sum_{i\neq j}(-1)^i x_{ij} z_{\hat{i}}$, $v_j=\sum_{i\neq j}(-1)^iy_{ij}z_{\hat{i}}$, $\delta_1=\Delta(12,34)$, $\delta_2=\Delta(13,24)$, and $\delta_3=\Delta(14,23)$, then the matrices of differentials $d_2$ and $d_3$ become
\begin{align*}
d_2(t)&=\setlength\arraycolsep{2pt} \begin{bmatrix*}
v_1 & u_1 & -\delta_1+\delta_2-\delta_3+t & 2\Delta(13,14) & -2\Delta(12,14) & -2\Delta(12,13) \\
-v_2 & -u_2 &-2\Delta(23,24) & -\delta_1-\delta_2+\delta_3+t & 2\Delta(12,24) & 2 \Delta(12,23) \\
v_3 & u_3 & 2\Delta(23,34)  & 2\Delta(13,34) & -\delta_1-\delta_2-\delta_3-t & -2\Delta(13,23)\\
v_4 & u_4 &2\Delta(24,34) & 2\Delta(14,34) & -2\Delta(14,24) &\delta_1-\delta_2-\delta_3+t\\
0 & 0&-z_{2,3,4} & -z_{1,3,4} & z_{1,2,4} & z_{1,2,3}
\end{bmatrix*},\\
d_3(t)&=\begin{bmatrix*}
b+t & 2a\\
-2c & -b+t\\
-v_1& -u_1\\
v_2 & u_2\\
v_3 & u_3\\
-v_4 & -u_4
\end{bmatrix*}.
\end{align*}

Then, $\im[d_1(t)]^T=\ker([d_2(t)]^T)$, where $[d_2(t)]^T$ and $[d_1(t)]^T$ are  transposes of matrices $d_2(t)$ and $d_1(t)$, respectively. Thus we get
$$J(t)=\im(d_1(t))=\langle -u_{2,3,4}+tz_{2,3,4},-u_{1,3,4}+tz_{1,3,4},-u_{1,2,4}+tz_{1,2,4}, u_{1,2,3}-tz_{1,2,3},u-t^2\rangle.$$

\begin{theorem}
The ideal $J(t)$ is a perfect ideal of codimension three in $B$. The minimal free resolution of $B/J(t)$ over $B$ is
\[0\xrightarrow{}B(-7)^2\xrightarrow{d_3(t)} B(-5)^6\xrightarrow{d_2(t)}B(-4)\oplus B(-3)^4 \xrightarrow{d_1(t)}B\]
\end{theorem}
\begin{proof}
We use Buchsbaum-Eisenbud exactness criterion \cite{DABDEs73}. The rank conditions are obviously satisfied. So we need to show that $\depth(I(d(t)_i))\geq 3$ for $1\leq i\leq 3$. But we see immediately by construction that the ideals $(I(d(t)_i),t)$ are equal to $(I(d_i),t)$. So they have depth four. Thus the claim follows.\end{proof}


\section{The ideals $J$ and $J(t)$ are in the linkage class of complete intersections}

Next we obtain sequences of links that link the ideals $J$ and $J(t)$ to complete intersection ideals, respectively.

\begin{theorem}\label{ciJ}
The ideals $J$ and $J(t)$ are in the linkage class of complete intersections.
\end{theorem}

\begin{proof}
First we show that the ideal $J$ is in the linkage class of a complete intersection by finding a regular sequence of elements of degree three in $J$. A good choice is taking
$L_1=\langle u_{1,2,3},u_{1,2,4},u_{1,3,4}\rangle$. If we set $K_1=(L_1:J)$, then 
$$K_1=\langle v_1,u_1,p_{K_1},q_{K_1},r_{K_1}\rangle$$ where 
\begin{align*}
p_{K_1}&=(\Delta(34,12)+\Delta(13,24)+\Delta(23,14))z_{1,3,4}+2\Delta(14,13)z_{2,3,4}),\\
q_{K_1}&=(\Delta(34,12)+\Delta(13,24)+\Delta(23,14))z_{1,2,4}+2\Delta(14,12)z_{2,3,4}, \; \text{and} \\
r_{K_1}& =(\Delta(34,12)+\Delta(13,24))z_{1,2,3}+\Delta(23,13)z_{1,2,4}+\Delta(12,23)z_{1,3,4}+2\Delta(13,12)z_{2,3,4}.
\end{align*}
Now let $N_1=\langle v_1,u_1,p_{K_1}\rangle$ and set $P_1=(N_1:K_1)$. Then 
$$P_1=\langle z_{1,3,4}, y_{1,4}z_{1,2,3}-y_{1,3}z_{1,2,4},x_{1,4}z_{1,2,3}-x_{1,3}z_{1,2,4},\Delta(14,13)\rangle,$$ and a resolution of $A/P_1$ is 
\[0\rightarrow A^2\xrightarrow{d_3^{P_1}}A^5\xrightarrow{d_2^{P_1}}A^4\xrightarrow{d_1^{P_1}} A\rightarrow A/P_1\rightarrow 0\] where $d_2^{P_1}=\begin{bmatrix*}
0& 0& y_{1,4}z_{1,2,3}-y_{1,3}z_{1,2,4} & x_{1,4}z_{1,2,3}-x_{1,3}z_{1,2,4} & \Delta(14,13)\\
x_{1,4} & x_{1,3} & -z_{1,3,4} & 0 & 0\\
-y_{1,4} & -y_{1,3} & 0 & -z_{1,3,4} &0\\
z_{1,2,4} & z_{1,2,3} & 0 & 0 & -z_{1,3,4}
\end{bmatrix*}
$.\\

\noindent Note that $d_2^{P_1}$ has three Koszul relations, and $A/P_1$ is a hyperplane section of a codimension two determinantal ideal $\langle  y_{1,4}z_{1,2,3}-y_{1,3}z_{1,2,4},x_{1,4}z_{1,2,3}-x_{1,3}z_{1,2,4}, \Delta(14,13) \rangle$. The ideal $P_1$ is well known to be in the linkage class of a complete intersection.



Similarly,  we show that the ideal $J(t)$ is in the linkage class of a complete intersection by finding a regular sequence of elements of degree three in $J(t)$. This time we consider $L_1(t)=\langle -u_{1,2,3}+z_{123}t,u_{1,2,4}-z_{124}t,u_{1,3,4}-z_{134}t\rangle$ and set $K_1(t)=(L_1(t):J(t))$.  Then 
$$K_1(t)=\langle v_1,u_1,p_{K_1}+z_{134}t,q_{K_1}+z_{124}t,r_{K_1}+z_{123}t\rangle,$$ 
Let $N_1(t)=\langle v_1,u_1,p_{K_1}+tz_{134}\rangle$ and $P_1(t)=(N_1(t):K_1(t))$. Then 
$$P_1(t)=\langle z_{1,3,4}, y_{1,4}z_{1,2,3}-y_{1,3}z_{1,2,4},x_{1,4}z_{1,2,3}-x_{1,3}z_{1,2,4},\Delta(14,13)\rangle,$$ and a resolution of $B/P_1(t)$ is 
\[0\rightarrow B^2\xrightarrow{d_3^{P_1(t)}}B^5\xrightarrow{d_2^{P_1(t)}}B^4\xrightarrow{d_1^{P_1(t)}} B\rightarrow B/P_1(t)\rightarrow 0\] where $d_2^{P_1(t)}=d_2^{P_1}$.

\noindent Here $d_2^{P_1(t)}$ also has three Koszul relations, and $B/P_1(t)$ is a hyperplane section of the codimension two determinantal ideal $\langle  y_{1,4}z_{1,2,3}-y_{1,3}z_{1,2,4},x_{1,4}z_{1,2,3}-x_{1,3}z_{1,2,4}, \Delta(14,13) \rangle$. The ideal $P_1(t)$ is well known to be in the linkage class of a complete intersection.\end{proof}



\section{Background on generic rings}
In this section, we work with finite free resolutions 
$${\bf F}_\bullet: 0\rightarrow F_3\rightarrow F_2\rightarrow F_1\rightarrow F_0$$
of length three over Noetherian rings $R$. If $\rank(F_i)=f_i$, we say that this resolution has format $(f_3, f_2, f_1, f_0)$.
If $\rank(d_i)=r_i$, we should have $f_3=r_3$, $f_2=r_3+r_2$, $f_1=r_2+r_1$. We will be dealing with the resolutions of cyclic modules, i.e., with the case $r_1=\rank(F_0)=1$.

In \cite{JWm16}, the third author constructed the generic rings ${\hat R}_{gen}$ for resolutions of all formats $(r_3, r_3+r_2, r_2+r_1, r_1)$.
This ring was related to Kac-Moody Lie algebra corresponding to the graph $T_{p,q,r}$, where 
$$(p,q,r)=(r_1+1, r_2-1, r_3+1).$$

 In particular, the ring ${\hat R}_{gen}$ was Noetherian if and only if $T_{p,q,r}$ was the Dynkin graph. We call the corresponding formats {\it the Dynkin formats}.
 
 The decomposition of the ring ${\hat R}_{gen}$ to irreducible representations of the group $\prod_{i=0}^3 GL(F_i)$ was described in \cite{JWm16}.
 The generators of ${\hat R}_{gen}$ are expected to come from three graded representations $W(d_3)$, $W(d_2)$, $W(a_2)$ (see \cite{CVW18}).
 
The natural question that arises is the description of the open set $U_{CM}$ of points in $Spec ({\hat R}_{gen})$ of points $P$ such that the corresponding ideals are perfect, i.e., the points that the dual complex ${\bf F}^*_\bullet$ is acyclic. One hopes that the points in the open set $U_{CM}$ give a generic form of perfect ideals with the resolution of the format $(2,6,5,1)$.

In \cite{CVW18}, the authors formulated a conjecture for the description of the open set $U_{CM}$ for all Dynkin types.
It says that three top graded components of representations $W(d_3)$, $W(d_2)$, $W(a_2)$ can be thought of as differentials of a complex ${\bf F}^{top}_\bullet$ over ${\hat R}_{gen}$ and that a point $P$ in $Spec ({\hat R}_{gen})$ is in $U_{CM}$ if and only if the complex ${\bf F}^{top}_\bullet$ tensored with $({\hat R}_{gen})_P$ is split exact. We call this open set in $Spec ({\hat R}_{gen})$ $U_{split}$.


\section{The ideals $J(t)$ and the open set $U_{split}$}

In the present paper, we work out for the format $(2,6,5,1)$, the form of the resolution that one gets over a point of $U_{split}$.
We recall (see \cite{KHLJW18}) the critical representations decompose as follows:
\begin{align*}W(d_3)&= F_2^*\otimes [F_3\oplus\bigwedge^2 F_1\oplus F_3^*\otimes \bigwedge^4F_1\oplus\bigwedge^2F_3^*\otimes S_{2,1^4}F_1],\\
W(d_2) &= F_2\otimes [F_1^*\oplus F_3^*\otimes F_1\oplus \bigwedge^2F_3^*\otimes\bigwedge^3F_1\oplus S_{2,1}F_3^*\otimes\bigwedge^5 F_1],\\
W(a_2) &= [F_1\oplus F_3^*\otimes\bigwedge^3F_1\oplus (\bigwedge^2F_3^*\otimes\bigwedge^4F_1\otimes F_1\oplus S_2F^*_3\otimes \bigwedge^5 F_1)\oplus\\
&\;\;\;\; \oplus S_{2,1}F_3^*\otimes S_{2,2,1,1,1}F_1\oplus S_{2,2}F_3^*\otimes S_{2^4,1}F_1].
\end{align*}

To calculate the split form, one assumes that the original complex ${\bf F}^{gen}_\bullet$ splits and one computes all higher structure theorems for this complex from $W(d_3)$, $W(d_2)$, $W(a_2)$, working over a polynomial ring in the defect variables, i.e., variables forming a basis of defect algebra for this case (see \cite{JWm16}).

In the following, $v^{(j)}_i$ ($1\le j\le 3$) denotes the tensor corresponding to the $i$-th graded component of $W(d_j)$ ($j=2,3$) and $v^{(1)}_i$ denotes the $i$-th graded component of $W(a_2)$.
We also denote the basis of $F_3$ by $\lbrace g_1, g_2\rbrace$, basis of $F_2$ by $\lbrace f_1, \ldots, f_6\rbrace$, basis of $F_1$ by $\lbrace e_1, \ldots, e_5\rbrace$, basis of $F_0$ by $1$. Putting the original complex ${\bf F}^{gen}_\bullet$ in the split exact form, that is, $d_3(g_1)=f_5, d_3(g_2)=f_6$, $d_2(f_i)=e_i$ for $1\le i\le 4$, $d_1(e_5)=1$ with all other differential on basis vectors being $0$, we introduce the defect variables (which correspond to the basis of the defect Lie algebra ${\bf L}$). 
In our case, ${\bf L}={\bf L}_1\oplus {\bf L}_2$, with ${\bf L}_1= F_3^*\otimes \bigwedge^2 F_1$ and ${\bf L}_2=\bigwedge^2F_3^*\otimes\bigwedge^4 F_1$ (see \cite{JWm16}). It is natural to name defect variables for our splitting form as follows. In ${\bf L}_1$, we have the variables $b^1_{i,j}$, $b^2_{i,j}$ ($1\le i<j\le 5$), and, in ${\bf L}_2$, we have the variables $c_{i,j,k,l}$ for $1\le i<j<k<l\le 5$.
Next, one works out the structure theorems $v^{(3)}_i$  from $W(d_3)$. For example, the map $v^{(3)}_1$ gives us the following component of the multiplication structure on the complex ${\bf F}_\bullet$:
$$e_i\cdot e_j=b^1_{i,j}f_5+b^2_{i,j}f_6$$
for $1\le i<j\le 4$,
$$e_i\cdot e_5=f_j+b^1_{i,5}f_5+b^2_{i,5}f_6$$
for $1\le i\le 4$.

The map $v^{(3)}_2$ is the lifting of the cycle
 
$$ \begin{tikzcd}
0 \arrow[r] & \bigwedge^2 F_3  \arrow[r] & F_3\otimes F_2  \arrow[r] &  S_2 F_2  \arrow[r] & S_2F_1 \\
& &  &  \bigwedge^4 F_1 \arrow[u, "S_2(v^{(3)}_1)"'] \arrow[ul, "v^{(3)}_2"] & 
\end{tikzcd}$$
where we have to add the part involving defect variables $c_{i,j,k,l}$. 

Finally, the map $v^{(3)}_3$ is the lifting

$$ \begin{tikzcd}
0 \arrow[r] & \bigwedge^2 F_3\otimes F_2  \arrow[r] &   F_3\otimes S_2 F_2  \arrow[r] &S_3 F_2 \\
& &  S_{2,1^4} F_1 \arrow[u, "(v^{(3)}_2)(v^{(3)}_1)"'] \arrow[ul, "v^{(3)}_2"] & 
\end{tikzcd}$$

One finds out that the tensor $v^{(3)}_3$  gives a matrix in which after the row and column operations all variables $b^1_{i,j}$ and $b^2_{i,j}$ for which the index $j$ is equal to $5$ are redundant. The remaining variables are renamed. The variables $b^1_{i,j}$ ($1\le i<j\le 4$) are named $x_{i,j}$, the variables $b^2_{i,j}$ ($1\le i<j\le 4$) are named $y_{i,j}$, the variables $c_{i,j,k,5}$ ($1\le i<j<k\le 4$) are named $z_{i,j,k}$ and the variable $c_{1,2,3,4}$ is named $t$. The middle matrix $d_2$ of the resolution of the deformed ideal $J(t)$ given above  is then the matrix one gets by working out the top graded component of $W(d_3)$. Using Macaulay 2, one can get the other two matrices $v^{(2)}_3$ and $v^{(1)}_4$.

One finds out that the tensor $v^{(3)}_3$  gives a matrix in which after the row and column operations all variables $b^1_{i,j}$ and $b^2_{i,j}$ for which the index $j$ is equal to $5$ are redundant. The remaining variables are renamed. The variables $b^1_{i,j}$ ($1\le i<j\le 4$) are named $x_{i,j}$, the variables $b^2_{i,j}$ ($1\le i<j\le 4$) are named $y_{i,j}$, the variables $c_{i,j,k,5}$ ($1\le i<j<k\le 4$) are named $z_{i,j,k}$ and the variable $c_{1,2,3,4}$ is named $t$. The middle matrix $d_2$ of the resolution of the deformed ideal $J(t)$ given above  is then the matrix one gets by working out the top graded component of $W(d_3)$. Using Macaulay 2, one can get the other two matrices $v^{(2)}_3$ and $v^{(1)}_4$.

This approach implies the following result.

\begin{theorem} Let $R$ be a local Noetherian ring, with characteristic of $R/{\mathfrak m}$ different than two and $I$ be a perfect ideal of codimension three with a resolution of $R/I$ having the format $(1,5,6,2)$. 
Let us denote this resolution ${\Bbb G}_\bullet$. Assume that the complex ${\Bbb G}^{top}_\bullet$ is split exact. Then there exists a homomorphism $\psi: B\rightarrow R$
such that ${\Bbb G}_\bullet = {\Bbb F}(t)_\bullet\otimes_{B}R$.
\end{theorem}
\begin{proof} There exists a homomorphism $\phi: {\hat R}_{gen}\rightarrow R$ such that ${\Bbb G}_\bullet = {\Bbb F}^{gen}_{\bullet}\otimes_{{\hat R}_{gen}}R$.
Since ${\Bbb G}^{top}_\bullet$ is split exact, we can change bases in $G_3$, $G_2$, $G_1$ so the complex ${\Bbb G}^{top}_\bullet$ is in canonical form. Then the relations described above in ${\hat R}_{gen}$ imply existence of homomorphism $\psi$. It sends the variables in $B$ to their counterparts in $R$ (after the appropriate renaming indicated above).
\end{proof}

\section{Multiplicative structure on the resolution of $A/J$.}

In this section, we assume that $K$ is a field of characteristic zero. Since the group $GL(F)\times GL(G)$ is reductive, we know that all structure theorems on our resolution have to exist in an equivariant form.
We determine the structure theorems on the resolution of the module $A/J$. It turns out that the equivariant property determines the structure theorems completely.

We start with the multiplication structure.
The multiplication $F_1\otimes F_1\rightarrow F_2$ has nonzero components

{\small $\displaystyle \bigwedge^2 (S_{2,2,2,1}F\otimes S_{1,1}G)=S_{4,4,3,3}F\otimes S_{2,2}G(-4,-2)\rightarrow (\bigwedge^2 F\otimes G)\otimes S_{3,3,3,3}F\otimes S_{2,1}G(-3,-2),\\
S_{2,2,2,1}F\otimes S_{1,1}G\otimes S_{2,2,2,2}F\otimes S_{2,2}G(-6,-1)=S_{4,4,4,3}F\otimes S_{3,3}G(-6,-1)\rightarrow(\bigwedge^3 F)\otimes S_{3,3,3,2}F\otimes S_{2,2}G(-4,-1),$}

\noindent so they are determined (up to scalar). Also, all structure constants have positive degree.

The multiplication $F_2\otimes F_1\rightarrow F_3$ has nonzero components
\begin{center}
{\small  $\displaystyle S_{3,3,3,2}F\otimes S_{2,2}G\otimes S_{2,2,2,1}F\otimes S_{1,1}G(-6,-2)\rightarrow (\bigwedge^2 F\otimes G)\otimes S_{4,4,4,4}F\otimes S_{3,2}G(-5,-2),$\\}
{\small  $\displaystyle S_{3,3,3,3}F\otimes S_{2,1}G\otimes S_{2,2,2,1}F\otimes S_{1,1}G(-5,-3)\rightarrow (\bigwedge^3 F)\otimes S_{4,4,4,4}F\otimes S_{3,2}G(-5,-2),$\\}
{\small  $\displaystyle S_{3,3,3,3}F\otimes S_{2,1}G\otimes S_{2,2,2,2}F\otimes S_{2,2}G(-7,-2)\rightarrow (\bigwedge^4F\otimes S_2 G)\otimes S_{4,4,4,4}F\otimes S_{3,2}G(-5,-2),$} 
\end{center}
so all components are determined by the equivariant property. All structure constants are of positive degree.

Let us  calculate the component $v^{(3)}_2$.This is the map {\small $\displaystyle \bigwedge^4$}$F_1\rightarrow F_3\otimes F_2$ has nonzero components: 
\begin{center}
{\small $\displaystyle \bigwedge^4 (S_{2,2,2,1}F\otimes S_{1,1}G)(-8,-4)\rightarrow  S_{4,4,4,4}F\otimes S_{3,2}G\otimes S_{3,3,3,3}F\otimes S_{2,1}G(-8,-4),$}
\end{center}
so this component is a split monomorphism,
{\small $$\bigwedge^3 (S_{2,2,2,1}F\otimes S_{1,1}G)\otimes S_{2,2,2,2}F\otimes S_{2,2}G(-10,-3)\rightarrow (\bigwedge^2F\otimes G)\otimes S_{4,4,4,4}F\otimes S_{3,2}G\otimes S_{3,3,3,2}F\otimes S_{2,2}G(-9,-3),$$}
Finally, we calculate the nonzero component $v^{(3)}_3$. This is the map {\small $$S_{2,1,1,1,1}F_1\rightarrow\bigwedge^2 F_3\otimes F_2.$$}
One sees easily that {\small $\bigwedge^5 F_1=S_{9,9,9,9}F\otimes S_{6,6}G(-12,-4)$}. Similarly, we have 
{\small $$\bigwedge^2 F_3=S_{8,8,8,8}F\otimes S_{5,5}G(-10,-4).$$}
The nonzero component of $v^{(3)}_3$ is

{\small $$S_{11,11,11,10}F\otimes S_{7,7}G(-14,-5)\rightarrow S_{11,11,11,10}F\otimes S_{7,7}G(-14,-5),$$}
which is a split isomorphism. So the nonzero component $v^{(3)}_3$ is completely split, as expected.

\section{Application: The pencils of quartics.}

Let $S=K[X,Y,Z]$ be the polynomial ring in three variables. We assume that $K$ is a field of characteristic zero.
Let $f_1, f_2$ be two general homogeneous polynomials of degree $4$ in the dual variables $X^*, Y^*, Z^*$.
Then $f_1, f_2$ define the ideal $I_{f_1, f_2}$ of the polynomials in $X,Y,Z$ vanishing on $f_1, f_2$ (they act via differential operators).
One can check easily that for general choice of $f_1, f_2$ the module $S/J_{f_1, f_2}$ has a minimal graded resolution
$${\Bbb F}^{f_1, f_2}_\bullet :0\rightarrow S^2(-7)\rightarrow S^6(-5)\rightarrow S(-4)\oplus S^4(-3)\rightarrow S.$$

\begin{theorem} For general $f_1, f_2$, the complex ${\Bbb F}^{f_1, f_2}_\bullet$ is a specialization of the complex ${\Bbb F}_\bullet$.
\end{theorem}

\begin{proof} A standard calculation shows that when the complex ${\Bbb F}^{f_1; f_2}_\bullet$ has form as above,
the Poincar{\'e} polynomial of $S/J_{f_1, f_2}$ is $1+3t+6t^2+6t^3+2t^4$.
We notice that the last row of the differential $d_2$ in the complex ${\Bbb F}^{f_1, f_2}_\bullet$ consists of six linear forms. For general $f_1, f_2$ these forms will generate the maximal ideal $(X, Y, Z)$. This means after applying some column operations we can bring this row to the form
$(0,0,0,X,Y,Z)$. Let $L_{f_1, f_2}$ be the ideal generated by four cubics in $J_{f_1, f_2}$.
Then the Poincar{\'e} polynomial of $S/L_{f_1, f_2}$ is $1+3t+6t^2+6t^3+3t^4$.
This means the minimal free resolution of $S/L_{f_1, f_2}$ is
$${\Bbb G}^{f_1, f_2}_\bullet : 0\rightarrow S^3(-7)\rightarrow S^3(-5)\oplus S^3(-6)\rightarrow S^4(-3)\rightarrow S.$$
In fact we have the exact sequence
$$0\rightarrow K(-3)\rightarrow S/L_{f_1, f_2}\rightarrow S/J_{f_1, f_2}\rightarrow 0$$
which makes ${\Bbb F}^{f_1, f_2}_\bullet$ quasiisomorphic to the mapping cone of the map of complexes
$${\Bbb K}(-3)_\bullet\rightarrow {\Bbb G}^{f_1, f_2}_\bullet$$
where ${\Bbb K}_\bullet$ is the Koszul complex on $X, Y, Z$.

We know the structure of almost complete intersections \cite{CVW18}.
Thus we know that the map of the third differential of ${\Bbb G}_\bullet$ has to be
$$d'_3=\begin{bmatrix}0&c_{1,2}&c_{1,3}\\-c_{1,2}&0&c_{2,3}\\
-c_{1,3}&-c_{2,3}&0\\
u_{1,1}&u_{1,2}&u_{1,3}\\
u_{2,1}&u_{2,2}&u_{2,3}\\
u_{3,1}&u_{3,2}&u_{3,3}
\end{bmatrix}$$
where $u_{i,j}$ are linear forms and $c_{i,j}$ are quadratic forms.

Now we start with the generic complex ${\Bbb F}_\bullet$, make one on the $z$-variables, say $z_{2,3,4}$ to be zero and we see that the linear row of $d_2$ has the same structure. So taking $L$ to be the ideal generated by four cubics in $J(t)$ we see that the third differential of the minimal free resolution ${\bf G}_\bullet$ of
$A/L$ has the form

$$\scriptsize{d^G_3=\setlength\arraycolsep{3 pt} \begin{bmatrix}  0&-b &-a\\
b&0&-\Delta(23,14)+\Delta(24,13)-\Delta(34,12)\\
a&\Delta(23,14)-\Delta(24,13)+\Delta(34,12)&0\\
-z_{1,3,4}&y_{3,4}&x_{3,4}\\
-z_{1,2,4}&y_{2,4}&x_{2,4}\\
-z_{1,2,3}&y_{2,3}&x_{2,3}
\end{bmatrix}}$$
where $a=2(x_{1,2}z_{1,3,4}-x_{1,3}z_{1,2,4}+x_{1,4}z_{1,2,3})$, $b=2(y_{1,2}z_{1,3,4}-y_{1,3}z_{1,2,4}+y_{1,4}z_{1,2,3})$.

Thus we need to see that any matrix $d'_3$ can be obtained from the matrix $d^G_3$ by specialization. Without loss of generality, we can assume that the elements $z_{1,2,3}, z_{1,2,4}, z_{1,3,4}$ generate the maximal ideal $(X, Y, Z)$.
Then we notice the lower part of the matrix is a specialization of the matrix with entries in $A$. This part of the matrix is determined by $z_{i,j,k}$ and by $x_{i,j}$ and $y_{i,j}$ with $i,j\ne 1$. The quadrics $c_{1,2}$ and $c_{1,3}$ can be also obtained as specialization because they are contained in the ideal $(X, Y, Z)$.
So the remaining problem is how to get the remaining quadric in the proper form. We show that actually any quadric can be brought in that form.
In order to do that let us investigate how we can change the remaining quadric $\Delta(23,14)-\Delta(24,13)+\Delta(34,12)$ without changing the others. This can be achieved for example by changing $x_{1,3}$ and $x_{1,4}$ into $x_{1,3}+\alpha z_{1,2,3}$ and  $x_{1,4}+\alpha z_{1,2,4}$ respectively.
The total change in $\Delta(23,14)-\Delta(24,13)+\Delta(34,12)$ will be $\alpha (z_{1,2,4}y_{2,3}-z_{1,2,3}y_{2,4})$, i.e. a multiple of the $2\times 2$ minor of the lower block of $d_3^G$ on the first two columns. Similarly we can get the other $2\times 2$ minors on these columns as well as minors on the first and third column. So if these six columns are linearly independent over $K$, we can change $\Delta(23,14)-\Delta(24,13)+\Delta(34,12)$ to an arbitrary quadric. In order to see it happens on Zariski open set we need to exhibit a $3\times 3$ matrix of linear forms such that all $2\times 2$ minors in the first and second and on the first and third columns are linearly independent. It is enough to take the matrix
$$\begin{bmatrix} y&z&0\\
x&y&z\\
0&x&-y
\end{bmatrix}.$$
\end{proof}

\section{Acknowledgment} The authors thank Lars Christensen and Oana Veliche for helpful discussions. Experiments with the computer algebra software Macaulay2 (see \cite{GS}) have provided numerous valuable insights. The first and second author acknowledges the support of the fourth author for their visit in Fall 2017 which was funded by the Sidney Professorial Fund. Jerzy Weyman was partially supported by the Sidney Professorial Fund and the NSF grant DMS-1802067.

\bibliographystyle{amsplain}

\begin{thebibliography}{1}


  
  \bibitem{AEB87} Anne E. Brown,
  \emph{A structure theorem for a class of grade three perfect ideals},
  J. Algebra. \textbf{105} (1977), no.~3, 308--327. \MR{MR0873666}
  
  \bibitem{CVW18} Lars~W. Christensen, Oana~Veliche, Jerzy~Weyman,
  On the structure of almost complete intersections of codimension 3, in preparation,

  \bibitem{DABDEs73} David~A. Buchsbaum and David Eisenbud,
  \emph{What makes a complex exact},
   J. of Algebra. \textbf{25} (1973), 259--286. \MR{0314819}
   
\bibitem{DABDEs77} David~A. Buchsbaum and David Eisenbud,
  \emph{Algebra structures for finite free resolutions, and some
    structure theorems for ideals of codimension {$3$}},
  Amer. J. Math. \textbf{99} (1977), no.~3, 447--485. \MR{MR0453723}
 

\bibitem{CVW-2} Lars~Winther Christensen, Oana Veliche, and Jerzy
  Weyman, \emph{Linkage classes of grade three perfect ideals}, in
  preparation.
  

\bibitem{GS} Daniel R. Grayson and Michael E. Stillman, Macaulay 2, a software system for research in algebraic geometry, 
Available at \href{https://faculty.math.illinois.edu/Macaulay2/}{https://faculty.math.illinois.edu/Macaulay2/}.


\bibitem{H75} Hochster, M. , CBMS Regional Conference Series in
  Mathematics vol. 24, 1975.


\bibitem{HU} Craig Huneke, Bernd Ulrich, \emph{The structure of
    linkage}, Ann. of Math. (2) 126 (1987), no.2, 277-334. \MR{MR0908149}
    
    \bibitem{KHLJW18} Kyu-Hwan Lee, Jerzy Weyman, \emph{Some branching formulas for Kac-Moody Lie algebras}, 
    preprint, 2018,


\bibitem{PS} Christiane Peskine, Lucien Szpiro, Liaison des varieties
  algebriques I, Invent. Math., 26 (1974), 271-302.

\bibitem{JWt73} Junzo Watanabe, \emph{A note on {G}orenstein rings of
    embedding codimension three}, Nagoya Math. J. \textbf{50} (1973),
  227--232. \MR{MR0319985}

\bibitem{JWm89} Jerzy Weyman, \emph{On the structure of free
    resolutions of length {$3$}}, J.  Algebra \textbf{126} (1989),
  no.~1, 1--33. \MR{MR1023284}
  
  \bibitem{JWm03} Jerzy Weyman, \emph{Cohomology of Vector Bundles and Syzygies}, Cambridge Tracts in Mathematics \textbf{149},
  Cambridge University Press 2003. \MR{MR1988690}
  

\bibitem{JWm16} Jerzy Weyman, \emph{Generic free resolutions and root
    systems}, arXiv 1609.02083,
    
   \bibitem{JWm17} Jerzy Weyman, \emph{Generic free resolutions and root systems II}, in preparation.

\end{thebibliography}

\providecommand{\MR}[1]{\mbox{\href{http://www.ams.org/mathscinet-getitem?mr=#1}{#1}}}
\renewcommand{\MR}[1]{\mbox{\href{http://www.ams.org/mathscinet-getitem?mr=#1}{#1}}}
\providecommand{\href}[2]{#2}

\end{document}